\pdfoutput=1
 \documentclass[11pt]{amsart}
\usepackage{amsmath,amssymb}
\usepackage{graphicx}
\usepackage{amsfonts,amscd,latexsym,epsfig, psfrag}

\newcommand{\PD}{{\rm PD}}

\newcommand{\Cc}{{\mathcal C}}

\newcommand{\Jj}{{\mathcal J}}

\newcommand{\less}{{\smallsetminus}}

\newcommand{\om}{{\omega}}
\newcommand{\eps}{{\varepsilon}}

\newcommand{\ka}{{\kappa}}

\newcommand{\Si}{{\Sigma}}

\newcommand{\Ss}{{\mathcal S}}

\newcommand{\Ee}{{\mathcal E}}

\newcommand{\Vv}{{\mathcal V}}

\newcommand{\C}{{\mathbb C}}

\newtheorem{theorem}{Theorem}[section]

\newtheorem{lemma}[theorem]{Lemma}

\numberwithin{figure}{section}
\numberwithin{equation}{section}
\newcommand{\MS}{{\medskip}}

\newcommand{\NI}{{\noindent}}

 \title{Symplectic embeddings of $4$-dimensional ellipsoids:  Erratum}
 \author{Dusa McDuff}\thanks{partially supported by NSF grant DMS 0905191}
\address{Department of Mathematics,
 Barnard College, Columbia University
New York, USA}
\email{dusa@math.columbia.edu} 
\keywords{symplectic embedding, pairwise connect sum, symplectic inflation,}
\subjclass[2000]{53D35}
\date{March 1, 2015}

\begin{document}

\maketitle

The paper \cite{Mce}  gives necessary and sufficient conditions for one $4$-dimensional ellipsoid to embed symplectically in another.
Emmanuel Opshtein pointed out that one step in the proof of sufficiency (Theorem~3.11)  was unjustified.  When embedding an ellipsoid into a ball, the idea is to embed a small ellipsoid  into the ball, perform 
a blow up procedure to convert this ellipsoid into two chains of embedded spheres $\Ss$, and then \lq\lq inflate" normal to these spheres, a process that has the effect of deforming    the symplectic form $\om_0$ so as to increase the relative size of the configuration $\Ss$ and hence of the ellipsoid.  
The inflation process described 
in Claim 2 of the proof of
\cite[Theorem~3.11]{Mce} used an embedded $J$-holomorphic curve $C_A$ in an appropriate homology class $A$  that intersects all the spheres in $\Ss$ $\om_0$-orthogonally.  
However the existence of the  embedded curve $C_A$ was not properly established.
Although an embedded representative of $A$ must  exist for
a generic $\om$-tame almost complex structure $J$, in the ellipsoidal context 
 one must work with a {\it non-generic } $J$, one for which
all the embedded spheres in $\Ss$ are also holomorphic.  At the time of writing, it is still unclear 
whether or not such an embedded curve must always exist, a question discussed further in~\cite{MO}.  
This question is resolved in the affirmative when all spheres in $\Ss$ have zero first Chern class: 
cf. Biran~\cite[Lemma~2.2B]{Bir}.  However, this condition is usually not satisfied in our situation.\MS

Although this problem does not affect any of the statements in these papers,
there are three places where this problem affects the argument:

\begin{itemize}\item[(I)]  in the proof of \cite[Theorem~3.11]{Mce};
\item[(II)] in the proof of \cite[Corollary~1.6(i)]{Mce} that claims spaces of embeddings of
ellipsoids into ellipsoids are path connected; and 
\item[(III)]  in the application of these ideas in \cite[Lemma~2.18(i)]{Mc6}.
\end{itemize}
In Case~(I), it turns out to be easy to fill the gap using the more general inflation procedure developed by Li--Usher in \cite{LU}.
We assume below that 
 $(M,\om_0)$ is a closed symplectic $4$-manifold that contains
a collection $\Ss$ of symplectically embedded and $\om_0$-orthogonally intersecting surfaces $C_i,1\le i\le L,$ in classes $S_i$ and of self-intersections $S_i\cdot S_i = k_i$.   A component $C_i$ is called positive (resp. negative) if 
$k_i\ge 0$ (resp. $k_i< 0$.

\begin{lemma}\label{le:1} With $\Ss$ as above,  let $m_i\ge 0$ be integers such that the class
 $A: = \sum_{i=1}^L m_i S_i$ satisfies $A\cdot S_i\ge 0$ for all $i$.
Then for each $\ka>0$  there is a symplectic form $\om_\ka$ on $M$ such that
\begin{itemize}\item[(i)]   $[\om_\ka] = [\om_0] + \ka \PD(A)$, where $\PD(A)$ denotes the Poincar\'e dual of $A$;
\item[(ii)] the restriction of  $\om_\ka$ to each smooth component $C_i$  of  $\Ss$ is symplectic.
\end{itemize}
\end{lemma}
\begin{proof} 
First note that, given any symplectic form $\om$ satisfying (ii), one can use Gompf's pairwise sum construction as in \cite[Theorem~2.3]{LU} to construct a new symplectic form $\om_i'$ that still satisfies (ii) and lies in class $[\om]+ \eps \PD(S_i)$ provided only that
$$
\om(S_i) +\eps k_i =  \int_{C_i} \om\ +\  \eps \PD(S_i)   > 0.\qquad (*)
$$
Thus for negative components $C_i$ one must choose $\eps< \frac{\om(S_i)}{|k_i|}$, while $\eps$ can be arbitrary for positive components.  Notice also that $\om_i'(S_j) \ge \om(S_j)$ for all $i\ne j$.  In other words, inflating by $\eps$ along $C_i$ increases the size of all the curves in $\Ss$ except perhaps for $C_i$ itself.

Now choose $\eps_0$ so that 
$$
0\; <\ \eps_0\ <\  \min_{1\le i\le L,  m_i>0, k_i<0} \ \frac {\om_0(S_i)}{m_i\ |k_i|}.
$$
Then, 
the class $[\om_0] + \eps_0 m_i\PD(S_i)$ evaluates positively  on each $S_j, j=1,\dots,L$. 
Hence, for each $\ka\le \eps_0$ we can inflate by the amount $m_i\ka$  along each curve $C^S_i, i=1,\dots, L,$ in turn, to construct a form 
$\om_{\ka}$ satisfying (i)  and also (ii).  

We now repeat this process starting with $\om_{\eps_0}$ instead of $\om_0$. 
Note that, although  $\om_{\eps_0}$ is  nondegenerate on each  $C^S_i$, the area has been redistributed, concentrating near the points where $C^S_i$ intersects the other components $C^S_j$. However, because $A\cdot S_i\ge 0$, 
 the total area of $C^S_i$ does not decrease, i.e.  $\om_{\eps_0}(S_i)\ge \om_0(S_i)$ for all  $i$. Hence
 the
 next step can also be of size $\eps_0$ since it satisfies $(*)$ when $\om_0$ is replaced by $\om_{\eps_0}$.
Continuing in this way, we may  therefore  construct a suitable form $\om_\ka$ for any given $\ka$ by
 a finite number of such steps.
\end{proof}

In the application in \cite{Mce}, when one is embedding one ellipsoid into another, 
one starts with a singular set $\Ss'$ in a blow up of $\C P^2$ that consists
 of four chains of spheres, two from the inner approximation to the larger ellipsoid and two from the outer approximation to the smaller ellipsoid. We need to inflate along an integral class $A': = qA$ with $(A')^2>0$ as described in \cite[Theorem~3.11]{Mce}. Here $A$ is chosen so that   
 \begin{itemize}\item $A\cdot S_i = 0$ for all $C_i$ in $\Ss'$,
 \item
 $A\cdot E\ge 0, E\in \Ee,$ where $\Ee$ is the set of exceptional classes, 
 i.e. classes that can be represented by symplectically embedded spheres of self-intersection $-1$.
   \end{itemize}
Given any $\om$-tame $J$ such that the curves in $\Ss'$ have $J$-holomorphic representatives, 
Taubes--Seiberg--Witten theory implies that $A'$ has some connected $J$-holomorphic nodal representative $\Si^{A'}$; see
 \cite[Lemma~3.1.1]{MO}.
 If we choose a  generic $\om$-tame $J$ with holomorphic restriction to $\Ss'$,
 then all the components of $\Si^{A'}$ that do not coincide with a component of $\Ss'$ must 
 be multiple covers of curves that are regular in the Gromov--Witten sense.
 Hence  if we also assume that
$J$ is integrable near $\Ss'$, then 
by standard arguments as described in \cite{LU}
  we may alter $J$ away from $\Ss'$ and perturb  the nodal curve
  to obtain a $J'$-holomorphic nodal representative  $\Si^{A'}$
  of $A'$ such  that
  \begin{itemize} 
  \item each smooth component of $\Si^{A'}$   is a multiple cover of an  embedded curve that  either is a curve in $\Ss'$, or  is an exceptional sphere,  or has  nonnegative self-intersection; in all cases it has nonnegative intersection with $A'$;
  \item all intersections of  these curves with each other and with  $\Ss'$ are $\om$-orthogonal.  
  \end{itemize}
For a detailed proof of this statement, see Case 1 of the proof of \cite[Prop.~3.1.3]{MO}.
It follows that, if we define $\Ss$ to consist of all the spheres in $\Ss'$ together with the embedded curves underlying the components of $\Si^{A'}$,  the class $A'$ has a representative of the form considered in Lemma~\ref{le:1}.   This fills the gap in Case~(I).
Proposition~5.1.2 in \cite{MO} is a more elaborate version of this inflation result.
 
 In Case~(II), one must check that this argument applies when the exceptional set $\Ss$ and class $A'$ are fixed, but one starts with a family of symplectic forms $\om_t, t\in [0,1],$ each of which is nondegenerate on the spheres in 
 $\Ss$.  For this it suffices to show  that a generic choice of $\om_t$-tame 
 $J_t$, holomorphic near $\Ss$, can be perturbed on $M\less \Ss$ so that
 there is a suitable $1$-parameter family of $J_t'$-holomorphic
 nodal curves $\Si^{A'}_t$ whose topological type does not vary with $t$.  Since bubbling for
  closed curves occurs in codimension $\ge 2$,  this is always possible. For more details, see~\cite[Thm~1.2.12]{MO}.

 Finally  Case~(III) concerns the following situation.  
 The set of spheres $\Ss$ (called $\Cc$ in \cite{Mc6}) is a union of spheres of self-intersection $-2$ lying in a four-manifold called $X_k$, and we are concerned with finding suitable embedded representatives of  a particular  class $E_3\in H_2(X_k)$ that for generic $J$ would be represented by an exceptional sphere.
In \cite{Mc6} we fix a suitable neighbourhood $\Vv$ of $\Ss$ and consider
 the space  $\Jj_\Vv(\Ss)$ of  tame almost complex structures  that are fixed near $\Ss$,
and its subset  $\Jj_{\Vv,reg}(\Ss)$ consisting
of all $J\in \Jj_\Vv(\Ss)$ for which every $J$-holomorphic sphere $u:S^2\to X_k$ whose image intersects  $X_k\less \Vv$
lies in a class $B$ with $c_1(B)>0$.  
We claim in \cite[Lemma~2.18~(i)]{Mc6}  that  $\Jj_{\Vv,reg}(\Ss)$ is  a residual and 
path-connected subset of $ \Jj_\Vv(\Ss)$, that in addition,  
has  the property that  
the class $E_3$ always has a smooth $J$-holomorphic representative.
The proof was inadequate,  particularly as regards the last claim.
However, in \cite[Prop~3.1.6]{MO} we prove 
that the subset $\Jj_{emb}(\Ss,A)$ of $J\in \Jj_\Vv(\Ss)$ for which $A: = E_3$ has an embedded representative
is both residual and path-connected.  Moreover, in the course of the proof we show that
 $J\in \Jj_\Vv(\Ss)$ belongs to $\Jj_{emb}(\Ss,A)$ provided that every $J$-holomorphic sphere 
not contained in $\Vv$ lies in a class $B$ with $d(B)\ge 0$, where 
$d(B): = B^2 + c_1(B)$ is the  Seiberg--Witten degree.  The   
adjunction inequality $1 + \frac 12(B^2- c_1(B))\ge 0$ for spheres gives $B^2\ge -1$, so that
the condition $c_1(B)>0$ implies $d(B)\ge 0$.
Hence $\Jj_{\Vv,reg}(\Ss)\subset \Jj_{emb}(\Ss,A)$.  Further, $\Jj_{\Vv,reg}(\Ss)$ is open and,
as we note in 
the proof of \cite[Lemma~2.18]{Mc6}, contains the path-connected, residual set
that is called $\Jj_{semi}(\Ss,A)$ in \cite{MO}.  Hence the proof given in \cite{MO} that
$\Jj_{emb}(\Ss,A)$ is path connected applies also to show that $\Jj_{\Vv,reg}(\Ss)$ is path-connected.

Because $\Ss$ consists of spheres with $c_1 = 0$ one could also prove the above statements using 
more numerical arguments,
 as in  Biran~\cite[Lemma~2.2B]{Bir}.  However, his result concerns only the existence of single 
embedded representatives,
not $1$-parameter families of such, and so completing the proof using this approach would also entail some amount of work.

\MS

Other more minor corrections to \cite{Mce}:
\MS

\NI p2:  Since $B(n)$ denotes the ball of radius  $\sqrt{n}$, for consistency $\C P^2(\mu)$ should denote the complex projective plane with the Fubini-Study form normalized so that the area of a line is $\pi \mu$.

\NI In Corollary 1.2, $v(k)$ denotes the packing constant for embeddings of $E(1,k)$ into $B(\sqrt k): = E(\sqrt k, \sqrt k)$.

\NI p 3.    The volume capacity $V$ of an ellipsoid is its volume.

\NI p 4:  In first sentence of paragraph $2$ after Prop. 1.9 replace  $K\cdot E = 1$ by $K\cdot E = -1$.

\MS

\NI {\bf Acknowledgement.}  I wish to thank Emmanuel Opshtein very warmly for pointing out the gap in \cite{Mce}, and for working with me on the construction of embedded curves, and also the referee of this erratum for a very careful reading.

\bibliographystyle{alpha}

\end{document}